\documentclass[12pt,a4paper,reqno]{amsart}
\usepackage{latexsym,amsmath,amssymb,amsthm}

\theoremstyle{plain}
 
\newtheorem{theorem}{Theorem}[section]

\newtheorem{lemma}[theorem]{Lemma}
\theoremstyle{definition}

\newtheorem{remark}[theorem]{Remark}
\newtheorem*{remark*}{Remark}

\def\r{\mathbb{R}}

\title[Minimum principles for a class of elliptic problems]{Some minimum principles for a class of nonlinear elliptic problems in divergence form}
 
\author{Cristian Enache}
\address{Department of Mathematics and Statistics, American University of Sharjah, University City Road, Sharjah, 26666 U.A.E.}
\email{cenache@aus.edu}

\author[Rafael L\'opez]{Rafael L\'opez}
\address{Departamento de Geometr\'{\i}a y Topolog\'{\i}a\\
Universidad de Granada 18071 Granada, Spain}
\email{rcamino@ugr.es}



\begin{document}

\begin{abstract} 
In this paper we study a general class of nonlinear elliptic problems in divergence form. First, for a natural subclass of coefficients, we prove that a suitable transformation of the solution is strictly concave when the given domain is strictly convex. Then, making use of this convexity property, we develop some minimum principles for an appropriate $P$-function, in the sense of L.~E.~Payne. Finally, these new minimum principles are applied to find a priori estimates for the solutions, in terms of the mean curvature of the boundary of the underlying domain. 
\end{abstract}

\maketitle

\section{Introduction and motivation}

In this paper we develop minimum principles and a priori estimates for solutions of a class of quasilinear elliptic equations in divergence form. Before stating our main results, we briefly describe the geometric motivation behind these equations.

\subsection{Geometric background}

In the seminal paper on quasilinear elliptic equations \cite{Se69}, Serrin considered the Dirichlet problem for a large class of second-order quasilinear elliptic equations. The most illustrative example corresponds to  
\begin{equation}\label{eq:Serrin-original}
\left((1+|\nabla u|^2)I-\nabla u \nabla u\right)\nabla^2u =nH(1+|\nabla u|^2)^\frac{\theta}{2}, 
\end{equation}
where $H$ and $\theta$ are constants (see \cite[p.~477]{Se69}). Here $u=u(x)$ is a function defined in a domain of $n$-dimensional Euclidean space $\r^n$, while $\nabla u$ and $\nabla^2u$ denote the gradient and the Hessian matrix of $u$, respectively. The most important case is $\theta=3$, where the equation describes hypersurfaces $x_{n+1}=u(x)$, $x\in\Omega\subset\r^n$, with constant mean curvature $H$. If $\theta<3$, the equation satisfies a maximum principle. Serrin made a deep study of the existence of solutions to the Dirichlet problem when $\theta\leq 3$ (see \cite[p.~478]{Se69}). 

For our convenience, we rewrite \eqref{eq:Serrin-original} in divergence form as
\begin{equation}\label{eq:1.1}
\mathrm{div}\left( \frac{\nabla u}{\sqrt{1+|\nabla u|^{2}}}\right) =nH\left(1+|\nabla u|^{2}\right)^{\frac{\theta-3}{2}}.
\end{equation}
We note that equation \eqref{eq:1.1} with arbitrary $\theta$ also appears in the study of translating solitons of the power mean curvature flow \cite{an,sh,sh2}. Equations of this type also arise in various geometric settings, which provides additional motivation; however, the focus of this paper is on analytic properties of solutions.

\subsection{The equations under study}

In this paper we focus on the case $\theta=0$ in equation \eqref{eq:1.1}. With this choice, the exponent on the right-hand side becomes $
(\theta-3)/2 = -3/2$. Taking $nH=1$, we obtain the \textbf{Euclidean mean-curvature equation}:
\begin{equation}\label{eq:1.2}
\mathrm{div}\left( \frac{\nabla u}{\sqrt{1+|\nabla u|^{2}}}\right) = \frac{1}{(1+|\nabla u|^{2})^{3/2}}.
\end{equation}

\begin{remark}
Our main results (Theorems~\ref{t1}, \ref{t2}, and \ref{t3}) concern the analytic properties of solutions (namely, convexity, minimum principles, and a priori estimates) and are established independently of any geometric interpretation. The geometric connection serves only as motivation for studying these particular equations.
\end{remark}

\subsection{The general problem}

The Dirichlet problem for \eqref{eq:1.2} is a particular case of the more general nonlinear elliptic problem in divergence form:
\begin{align}
\mathrm{div}\left( g\left( \left\vert \nabla u\right\vert ^{2}\right) \nabla u\right) &= f(u)\, G\left( \left\vert \nabla u\right\vert ^{2}\right) \quad \text{in }\Omega, \label{eq:1.4} \\
u &= 0 \quad \text{on }\partial \Omega, \label{eq:1.5}
\end{align}
where, throughout this paper, $\Omega \subset \mathbb{R}^{n}$, $n\geq 2$, is a smooth bounded strictly convex domain, and $f$ and $g$ are positive $C^{2}$-functions satisfying
\[
G(s) := g(s) + 2s\,g^{\prime}(s) > 0 \quad \text{for } s > 0.
\]
As a particular case, setting $f(s) = 1$ and $g(s) = (1+s)^{-1/2}$ in \eqref{eq:1.4} yields equation \eqref{eq:1.2}.

\subsection{Main results}

Our first main result establishes a convexity property for an appropriate transformation of the solution. {Since the right-hand side of \eqref{eq:1.4} is positive and $u = 0$ on $\partial\Omega$, the strong maximum principle yields $u < 0$ in $\Omega$. Moreover, by Hopf's boundary point lemma, $u_{\mathbf{n}} > 0$ on $\partial\Omega$, where $\mathbf{n}$ denotes the exterior unit normal to $\partial\Omega$.} To state this result, we introduce the following definitions. For $y < 0$, let
\begin{equation}\label{eq:1.6}
F(y) := \int_{y}^{0} f(s)\, ds.
\end{equation}
Since $f > 0$ and we integrate from $y < 0$ to $0$, we have $F(y) > 0$ for all $y < 0$. {Since $u < 0$ in $\Omega$,} we then define
\begin{equation}\label{eq:1.6b}
v = v(u) := \int_{u}^{0} \frac{dy}{\sqrt{F(y)}},
\end{equation}
where $u$ is the solution of problem \eqref{eq:1.4}--\eqref{eq:1.5}. Note that the integrand in \eqref{eq:1.6b} is well-defined and positive since $F(y) > 0$ for $y < 0$.

\begin{theorem}\label{t1}
Assume that $u$ is the solution of problem \eqref{eq:1.4}--\eqref{eq:1.5}, that $g$ satisfies $g^{\prime}(s)/G(s) = -\lambda/2$ for some constant $\lambda \geq 0$, and that $f$ satisfies the following properties:
\begin{equation}\label{eq:1.7}
f^{\prime} {\geq} 0, \qquad 2(f^{\prime})^{2} - f\,f^{\prime\prime} \geq 0.
\end{equation}
Then $v = v(u)$, defined in \eqref{eq:1.6}--\eqref{eq:1.6b}, is strictly concave in $\Omega$.
\end{theorem}

\begin{remark}
The condition $g^{\prime}/G = -\lambda/2$, with $G = g + 2s\,g^{\prime}$, determines $g$ explicitly. Indeed, $g^{\prime} = -\tfrac{\lambda}{2}\left(g + 2s\,g^{\prime}\right)$, so that
\[
(1+\lambda s)\,g^{\prime}(s) = -\frac{\lambda}{2}\,g(s),
\]
whence $g^{\prime}/g = -\tfrac{\lambda}{2}(1+\lambda s)^{-1}$ and $g(s) = C(1+\lambda s)^{-1/2}$. Under the normalization $g(0)=1$,
\[
g(s) = (1+\lambda s)^{-1/2}, \qquad G(s) = (1+\lambda s)^{-3/2}, \qquad \lambda \geq 0.
\]
The value $\lambda = 1$ gives the Euclidean mean-curvature coefficient $g(s) = (1+s)^{-1/2}$, while $\lambda = 0$ gives $g \equiv 1$, i.e.\ the semilinear equation $\Delta u = f(u)$.
\end{remark}

The case $g \equiv 1$ in the above theorem was proved by Caffarelli and Friedman~\cite{CF85} for $n=2$, and by Korevaar and Lewis~\cite{KL87} in higher dimensions. Both works used a deformation technique where a crucial role was played by a constant rank theorem. We follow their approach here. For more details and other techniques used to prove similar convexity results, we refer the reader to the surveys of McCuan~\cite{Mc02} and Ma--Ou~\cite{MO09}, as well as the book of Kawohl~\cite{Ka06}.

Next, using the above convexity result, we develop a minimum principle for an appropriate $P$-function in the sense of Payne (see the book of Sperb~\cite{Sp81}). We define
\begin{equation}\label{eq:1.8}
\Phi(\mathbf{x};\beta) := \left\vert \nabla u\right\vert^{2} {{}+{}} \beta\, F(u),
\end{equation}
where $u$ is the solution of problem \eqref{eq:1.4}--\eqref{eq:1.5}.

\begin{theorem}\label{t2}
If $\beta \in [1,2)$, the data of problem \eqref{eq:1.4}--\eqref{eq:1.5} satisfy
\[
f^{\prime}\,G + \beta\, f^{2}\,G^{\prime} \leq 0,
\]
and the transformed function $v = v(u)$, defined in \eqref{eq:1.6}--\eqref{eq:1.6b}, is strictly concave in $\Omega$, then $\Phi(\mathbf{x};\beta)$, defined in \eqref{eq:1.8}, attains its minimum value on $\partial\Omega$.
\end{theorem}

\begin{remark}
The strict concavity of $v$ assumed in Theorem~\ref{t2} plays the role of hypothesis (P) of Philippin and Safoui~\cite{PS04}: here the strictly monotone transform is $v = v(u)$ itself, and its strict concavity both validates the Philippin--Safoui inequality \eqref{eq:3.3} used in the proof and guarantees a unique interior critical point of $u$. For the class $g(s) = (1+\lambda s)^{-1/2}$ this concavity is provided by Theorem~\ref{t1}; the Euclidean case $\lambda = 1$ underlies Theorem~\ref{t3}.
\end{remark}

Finally, making use of the minimum principle for $\Phi(\mathbf{x};\beta)$, we derive the following a priori estimate for the solution of equation \eqref{eq:1.2}.

\begin{theorem}\label{t3}
Assume that $\Omega \subset \mathbb{R}^{n}$, $n \geq 2$, is a smooth bounded strictly convex domain. If $u$ is the solution of
\begin{equation}\label{eq:1.9}
\begin{cases}
\displaystyle \mathrm{div}\left( \frac{\nabla u}{\sqrt{1+|\nabla u|^{2}}}\right) = \frac{1}{(1+|\nabla u|^{2})^{3/2}} & \text{in } \Omega, \\[2mm]
u = 0 & \text{on } \partial\Omega,
\end{cases}
\end{equation}
then
\begin{equation}\label{eq:1.10}
-u_{\min} \geq \left( \sqrt[3]{\frac{\alpha}{2} + \sqrt{\frac{\alpha^{2}}{4} + \frac{1}{27}}} + \sqrt[3]{\frac{\alpha}{2} - \sqrt{\frac{\alpha^{2}}{4} + \frac{1}{27}}} \right)^{2},
\end{equation}
where $u_{\min} = \min_{\overline{\Omega}} u(\mathbf{x})$, $\alpha = 1/(2(n-1)K_{\max})$, and $K_{\max}$ is the maximum value of the mean curvature $K(s)$ of the boundary $\partial\Omega$.
\end{theorem}

A similar method to obtain a priori bounds for solutions of problems of physical or geometrical interest, based on minimum principles for $P$-functions, has been previously considered by several authors. We refer the reader to \cite{Ph79,PS04} and more recently \cite{En14, En15, EL19, Ma99, Ma00, PP12, To22}. To prove our new minimum principle and a priori estimates, we make use of the techniques from these papers, especially \cite{PS04}, in the case where the solution has a certain convexity property.

\subsection{Organization of the paper}

The paper is organized as follows. In Section~2 we prove the convexity result of Theorem~\ref{t1}, based on a constant rank theorem and a deformation technique. The proofs of Theorem~\ref{t2} and Theorem~\ref{t3} are given in Sections~3 and~4, respectively. In Section~5 we present some complementary results and remarks.

\subsection{Notation}

Throughout this paper we adopt the following notation:
\[
u_{i} = \frac{\partial u}{\partial x_{i}}, \qquad u_{ij} = \frac{\partial^{2}u}{\partial x_{i}\partial x_{j}}.
\]
We use the summation convention: summation over repeated indices. For example,
\[
u_{ij}u_{i}u_{j} = \sum_{i=1}^{n}\sum_{j=1}^{n} \frac{\partial^{2}u}{\partial x_{i}\partial x_{j}} \frac{\partial u}{\partial x_{i}} \frac{\partial u}{\partial x_{j}}.
\]

\section{Proof of Theorem \ref{t1}}

First, we note that from the definition of $v$ in \eqref{eq:1.6b}, we have
\[
\frac{dv}{du} = -\frac{1}{\sqrt{F}}, \qquad \frac{du}{dv} = -\sqrt{F}.
\]
Moreover, there exists a function $h$ such that $u = h(v)$, and
\begin{equation}\label{eq:2.2}
\begin{aligned}
u_{i} &= h^{\prime} v_{i}, \\
\left\vert \nabla u\right\vert^{2} &= (h^{\prime})^{2} \left\vert \nabla v\right\vert^{2}, \\
u_{ij} &= h^{\prime} v_{ij} + h^{\prime\prime} v_{i}v_{j}, \qquad i,j = 1,\ldots,n.
\end{aligned}
\end{equation}

Next, we find the equation satisfied by $v$ by substituting $u = h(v)$ into \eqref{eq:1.4}. We first write equation \eqref{eq:1.4} as
\begin{equation}\label{eq:2.3}
\left( \frac{g}{G}\delta_{ij} + 2\frac{g^{\prime}}{G}u_{i}u_{j}\right) u_{ij} = f(u),
\end{equation}
then use \eqref{eq:2.2} to rewrite \eqref{eq:2.3} as
\[
\left( \frac{g}{G}\delta_{ij} + 2\frac{g^{\prime}}{G}(h^{\prime})^{2}v_{i}v_{j}\right) \left( h^{\prime} v_{ij} + h^{\prime\prime} v_{i}v_{j}\right) = f(h(v)),
\]
or equivalently,
\begin{equation}\label{eq:2.5}
h^{\prime}\left( \frac{g}{G}\delta_{ij} + 2\frac{g^{\prime}}{G}(h^{\prime})^{2}v_{i}v_{j}\right) v_{ij} + h^{\prime\prime}\left( \frac{g}{G}|Dv|^{2} + 2\frac{g^{\prime}}{G}(h^{\prime})^{2}|Dv|^{4}\right) = f(h(v)).
\end{equation}
Using the facts that
\[
h^{\prime}(v) = \frac{du}{dv} = -\sqrt{F} < 0, \qquad h^{\prime\prime}(v) = \frac{d}{dv}\left( \frac{du}{dv}\right) = -\frac{F^{\prime}}{2\sqrt{F}}(-\sqrt{F}) = \frac{F^{\prime}}{2} = -\frac{f}{2},
\]
{together with $(h^{\prime})^{2} = F$, we simplify the scalar term in \eqref{eq:2.5}. Since $G = g + 2s\,g^{\prime}$ with $s = |\nabla u|^{2} = F|Dv|^{2}$, we have $G = g + 2F|Dv|^{2}g^{\prime}$, whence}
\begin{equation}\label{eq:2.6a}
{\frac{g}{G}|Dv|^{2} + 2\frac{g^{\prime}}{G}(h^{\prime})^{2}|Dv|^{4} = \frac{g + 2F|Dv|^{2}g^{\prime}}{G}\,|Dv|^{2} = |Dv|^{2}.}
\end{equation}
{Substituting $h^{\prime} = -\sqrt{F}$, $h^{\prime\prime} = -f/2$ and \eqref{eq:2.6a} into \eqref{eq:2.5} gives
\[
-\sqrt{F}\left( \frac{g}{G}\delta_{ij} + 2\frac{g^{\prime}}{G}F v_{i}v_{j}\right) v_{ij} - \frac{f}{2}|Dv|^{2} = f.
\]
Dividing this equation by $-f$, we obtain}
\begin{equation}\label{eq:2.7}
a_{ij}(v, Dv)\, v_{ij} = b(v, Dv),
\end{equation}
where
\[
a_{ij}(v, Dv) = \frac{\sqrt{F}}{f}\left( \frac{g}{G}\delta_{ij} + 2\frac{g^{\prime}}{G}F\,v_{i}v_{j}\right), \qquad b(v, Dv) = -1 - \frac{1}{2}|Dv|^{2}.
\]

\subsection{A constant rank theorem}

\begin{theorem}\label{thm:constant-rank}
Let $u \in C^{4}(\Omega)$ be a solution of problem \eqref{eq:1.4}--\eqref{eq:1.5}, where $\Omega \subset \mathbb{R}^{n}$, $n \geq 2$, is any domain, $g$ satisfies $g^{\prime}/G = -\lambda/2$ for some constant $\lambda \geq 0$, and $f$ satisfies conditions \eqref{eq:1.7}. If $v = v(u)$ defined in \eqref{eq:1.6}--\eqref{eq:1.6b} is a concave function, that is, the Hessian matrix $W = (v_{ij})$ of $v$ is negative semidefinite, then $W$ must have constant rank in $\Omega$.
\end{theorem}

\begin{remark*}
The restriction $g^{\prime}/G = -\lambda/2$ is what makes the two-dimensional computation below self-contained. For a general coefficient $g$,
\[
a_{22} = \frac{\sqrt{F}}{f} - 2\,v_{1}^{2}\,\frac{F\sqrt{F}}{f}\,\frac{g^{\prime}}{G},
\]
in which the ratio $g^{\prime}/G$ is evaluated at $s = F(v)|Dv|^{2}$ and hence depends on $v$. Differentiating $a_{22}$ twice in $v$ then produces additional chain-rule terms in $(g^{\prime}/G)^{\prime}$ and $(g^{\prime}/G)^{\prime\prime}$ which, once expanded, involve derivatives of $g$ up to third order; obtaining a definite sign for $a_{22,vv}$ in that generality would require considerably more complicated structural assumptions on $g$. The condition $g^{\prime}/G = -\lambda/2$ makes this ratio constant, removes those terms, and singles out the natural class $g(s) = (1+\lambda s)^{-1/2}$.
\end{remark*}

Note that the second condition in \eqref{eq:1.7} means that $1/f(u)$ is convex in $u$. We first give a self-contained proof for $n=2$, and then treat the case $n \geq 3$ by the microscopic convexity principle of Bian and Guan~\cite{BG10}.

\begin{proof}[Proof for dimension $n=2$]
{First, the minimal rank of $W$ cannot be zero: if $D^{2}v = 0$ at some point, then the left-hand side of \eqref{eq:2.7} vanishes there, contradicting $b = -1 - \tfrac{1}{2}|Dv|^{2} < 0$. Hence the minimal rank of $W$ is at least one.}
Suppose that $W$ attains its minimal rank~$1$ at some point $z_{0} \in \Omega$. We will prove that the rank of $W$ is identically equal to~$1$ in $\Omega$; otherwise, the rank equals~$2$ in $\Omega$. Define
\[
P(x) = \det W(x).
\]
We have $P(z_{0}) = 0$. We shall show that there exists an open neighborhood $\mathcal{O}$ of $z_{0}$ such that $P(x) = 0$ in $\mathcal{O}$.

If true, this implies that $U = \{x \in \Omega : P(x) = 0\}$ is open. On the other hand, $U$ is also closed since $P(x)$ is continuous on $\Omega$. Therefore, $P(x) = 0$ in $\Omega$, which means that $W$ has rank~$1$. Since $P \geq 0$ in $\Omega$ and $P(z_{0}) = 0$, it remains to prove
\begin{equation}\label{eq:2.11}
\sum_{i,j=1}^{2} {a_{ij}}P_{ij} \lesssim 0
\end{equation}
in a small open neighborhood $\mathcal{O}$ of $z_{0}$, and then apply the strong minimum principle to conclude $P(x) = 0$ in $\mathcal{O}$.

We use the classical notation from \cite{CF85} and \cite{KL87}. For two functions $p$ and $q$ defined in an open set $\mathcal{O} \subset \Omega$ and $y \in \mathcal{O}$, we write
\[
p(y) \lesssim q(y)
\]
when there exist positive constants $C_{1}$ and $C_{2}$ such that
\[
(p - q)(y) \leq \left( C_{1}\left\vert \nabla P\right\vert^{2} + C_{2}P\right)(y).
\]
We write $p \lesssim q$ if the above inequality holds in $\mathcal{O}$ with constants $C_{1}$ and $C_{2}$ independent of $y$ in this neighborhood. We also write
\[
p(y) \sim q(y) \quad \Leftrightarrow \quad p(y) \lesssim q(y) \text{ and } q(y) \lesssim p(y).
\]
{Thus $p \sim q$ means that $p - q$ is controlled, up to fixed multiplicative constants, by $|\nabla P|^{2}$ and $P$; in the course of the computation any term that is a bounded multiple of $|\nabla P|^{2}$ or of $P$ may therefore be discarded.}

To prove \eqref{eq:2.11} at an arbitrary point $z \in \mathcal{O}$, as in \cite{CF85}, we choose normal coordinates by performing an appropriate rotation about $z$ such that the Hessian matrix $W$ is diagonal at $z$ and $v_{11} \leq v_{22} \leq 0$ at $z$. All computations below are carried out at this fixed point $z$. Since $W$ is negative semi-definite of minimal rank one, near $z_{0}$ the eigenvalue $v_{11}$ is strictly negative and bounded away from $0$ (the good eigenvalue), while $v_{22}$ is the degenerate eigenvalue, with $v_{22} = 0$ at $z_{0}$. From $P = \det W = v_{11}v_{22} - v_{12}^{2}$, with $v_{12} = 0$ at $z$ and $v_{11}$ bounded away from $0$, differentiation gives
\[
0 \sim P \sim -v_{22}, \qquad 0 \sim P_{i} \sim -v_{22i}, \quad i = 1,2,
\]
so that the third derivatives $v_{22i}$ of the degenerate entry are negligible. Differentiating $P$ twice and evaluating at $z$,
\[
P_{ij} \sim v_{11}v_{22ij} - 2v_{12i}v_{12j}, \quad i,j = 1,2,
\]
where, since $v_{22i} \sim 0$, only the component $v_{121} = v_{112}$ survives in the quadratic term. Therefore
\begin{equation}\label{eq:2.19}
a_{ij}P_{ij} = v_{11}\,a_{ij}v_{22ij} - 2a_{11}v_{112}^{2}.
\end{equation}
Here and below, for the composed coefficients $a_{ij}(v,Dv)$ and $b(v,Dv)$ the symbol $D_{k}$ denotes the total derivative with respect to $x_{k}$ (by the chain rule, $D_{k}b = b_{v}v_{k} + b_{v_{l}}v_{lk}$). From equation \eqref{eq:2.7}, since $v_{22} \sim 0$, we have $a_{11}v_{11} \sim b$. Differentiating \eqref{eq:2.7} with respect to the degenerate direction $x_{2}$,
\[
D_{2}a_{ij}\,v_{ij} + a_{ij}v_{ij2} = D_{2}b,
\]
and using $v_{22} \sim 0$ and $v_{22i} \sim 0$, we obtain $a_{11}v_{112} \sim D_{2}b - v_{11}D_{2}a_{11}$. Differentiating once more with respect to $x_{2}$,
\begin{equation}\label{eq:2.23}
a_{ij}v_{22ij} \sim D_{22}b - v_{11}D_{22}a_{11} - 2v_{112}D_{2}a_{11}.
\end{equation}
Inserting \eqref{eq:2.23} into \eqref{eq:2.19} and multiplying by $a_{11}^{2}$,
\begin{equation}\label{eq:2.25}
a_{11}^{2}a_{ij}P_{ij} \sim a_{11}(a_{11}v_{11})D_{22}b - (a_{11}v_{11})^{2}D_{22}a_{11} - 2(a_{11}v_{11})D_{2}a_{11}(a_{11}v_{112}) - 2a_{11}(a_{11}v_{112})^{2}.
\end{equation}
We compute the total derivatives explicitly. As $b = -1 - \tfrac{1}{2}|Dv|^{2}$ depends only on $Dv$, $D_{k}b = -v_{l}v_{lk}$, and hence at $z$
\begin{equation}\label{eq:2.26}
D_{2}b = -v_{2}v_{22} \sim 0, \qquad D_{22}b = -\big(v_{12}^{2}+v_{22}^{2}\big) - v_{l}v_{l22} \sim 0,
\end{equation}
since $v_{22} \sim 0$ and $v_{22l} \sim 0$. Similarly, the gradient-variable contributions to $D_{2}a_{11}$ and $D_{22}a_{11}$ carry factors $v_{12} = 0$, $v_{22} \sim 0$ or $v_{22l} \sim 0$ and are therefore negligible, so that the reduction to partial derivatives in $v$ is justified:
\begin{equation}\label{eq:2.27}
D_{2}a_{11} \sim a_{11,v}v_{2}, \qquad D_{22}a_{11} \sim a_{11,vv}v_{2}^{2}.
\end{equation}
Substituting these, together with $a_{11}v_{11} \sim b$ and $a_{11}v_{112} \sim -v_{11}D_{2}a_{11}$, into \eqref{eq:2.25}, the two middle terms combine through the factor $b - a_{11}v_{11} = a_{22}v_{22} \sim 0$ and drop out, leaving
\[
a_{11}^{2}a_{ij}P_{ij} \sim -b^{2}a_{11,vv}v_{2}^{2}.
\]
It remains to show that $a_{11,vv} \gtrsim 0$.

Under $g^{\prime}/G = -\lambda/2$ we have $g/G = 1 + \lambda F|Dv|^{2}$, whence
\[
a_{11} = \frac{\sqrt{F}}{f}\left( \frac{g}{G} + 2\frac{g^{\prime}}{G}Fv_{1}^{2}\right) = \frac{\sqrt{F}}{f}\left( 1 + \lambda F v_{2}^{2}\right) = \frac{\sqrt{F}}{f} + \lambda\,v_{2}^{2}\frac{F\sqrt{F}}{f};
\]
the surviving complementary gradient component is $v_{2}^{2}$. We compute successively:
\[
\left( \frac{\sqrt{F}}{f}\right)_{v} = \frac{\frac{F^{\prime}}{2\sqrt{F}}f - \sqrt{F}f^{\prime}}{f^{2}}(-\sqrt{F}) = \frac{f^{2} + 2Ff^{\prime}}{2f^{2}} = \frac{1}{2} + \frac{Ff^{\prime}}{f^{2}},
\]
\[
\left( \frac{\sqrt{F}}{f}\right)_{vv} = \frac{F\sqrt{F}}{f^{3}}\left(2(f^{\prime})^{2} - ff^{\prime\prime} + \frac{f^{2}f^{\prime}}{F}\right),
\]
and
\[
\left( \frac{F\sqrt{F}}{f}\right)_{v} = \frac{3}{2}F + F\frac{Ff^{\prime}}{f^{2}},
\]
\[
\left( \frac{F\sqrt{F}}{f}\right)_{vv} = \frac{3}{2}f\sqrt{F} + \frac{F\sqrt{F}f^{\prime}}{f} + \frac{F^{2}\sqrt{F}}{f^{3}}\left(2(f^{\prime})^{2} - ff^{\prime\prime} + \frac{f^{2}f^{\prime}}{F}\right).
\]

Since $g^{\prime}/G$ is constant, differentiating $a_{11}$ in $v$ (with $Dv$ fixed) produces no chain-rule terms in $(g^{\prime}/G)^{\prime}$ or $(g^{\prime}/G)^{\prime\prime}$, and
\[
a_{11,vv} = \left(\frac{\sqrt{F}}{f}\right)_{vv} + \lambda\,v_{2}^{2}\left(\frac{F\sqrt{F}}{f}\right)_{vv}.
\]
By \eqref{eq:1.7} and the positivity of $f$ and $F$, the bracket $2(f^{\prime})^{2} - ff^{\prime\prime} + \frac{f^{2}f^{\prime}}{F} \geq 0$, so $\left(\sqrt{F}/f\right)_{vv} \geq 0$; likewise $\left(F\sqrt{F}/f\right)_{vv} \geq 0$, every term being non-negative under \eqref{eq:1.7}. As $\lambda \geq 0$ and $v_{2}^{2} \geq 0$, it follows that $a_{11,vv} \geq 0$. Hence $a_{11}^{2}\sum_{i,j}a_{ij}P_{ij} \sim -b^{2}a_{11,vv}v_{2}^{2} \leq 0$; since $a_{11}^{2} > 0$, this yields $\sum_{i,j}a_{ij}P_{ij} \lesssim 0$, i.e.\ \eqref{eq:2.11}. Applying the strong minimum principle to $P \geq 0$, which attains its minimum value $0$ at $z_{0}$, we conclude that $P \equiv 0$ in a neighborhood of $z_{0}$, completing the proof for $n = 2$.
\end{proof}
For $n \geq 3$, the conclusion follows from the microscopic convexity principle of Bian and Guan~\cite{BG10}. Indeed, after setting $w = -v$, equation \eqref{eq:2.7} takes the form
\[
\mathcal{F}(D^{2}w, Dw, w) = 0.
\]
This operator is elliptic. Moreover $\mathcal{F}(0, Dw, w) \neq 0$, and its associated zero sub-level set satisfies the structural assumptions of the microscopic convexity principle established in Bian and Guan \cite[Theorem~1.2]{BG10}. Therefore $D^{2}w$ has constant rank. Since $D^{2}v = -D^{2}w$, the Hessian of $v$ also has constant rank. Consequently, Theorem~\ref{thm:constant-rank} holds for every $n \geq 2$.

The remaining part of the proof of Theorem~\ref{t1} is standard, based on a deformation process. We first prove that $v$ is strictly concave in the ball.

\subsection{Concavity in the ball}

\begin{lemma}\label{lem:ball}
Let $B_{R}(0)$ be a ball of radius $R$ in $\mathbb{R}^{n}$, $n \geq 2$, and let $u$ be the solution of problem \eqref{eq:1.4}--\eqref{eq:1.5} in $\Omega = B_{R}(0)$. Then $v$, defined in \eqref{eq:1.6}--\eqref{eq:1.6b}, is a strictly concave function in $B_{R}(0)$.
\end{lemma}

\begin{proof}
From a seminal paper of Gidas, Ni, and Nirenberg~\cite{GNN79}, we know that $u$ is radially symmetric in the ball:
\[
u(x) = \varphi(r), \qquad r = |x|,
\]
for some function $\varphi(r)$ defined in $[0,R]$ with $\varphi(r) < 0$. Then $\varphi$ is an increasing function of $r$ in $(0,R)$ and $\varphi^{\prime}(0) = \varphi(R) = 0$. {Writing $g = g(\varphi^{\prime}(r)^{2})$ and $G = G(\varphi^{\prime}(r)^{2})$, the} function $\varphi$ satisfies the ordinary differential equation
\begin{equation}\label{eq:2.37}
G{(\varphi^{\prime 2})}\varphi^{\prime\prime} + (n-1)g{(\varphi^{\prime 2})}\frac{\varphi^{\prime}}{r} = f(\varphi)G{(\varphi^{\prime 2})},
\end{equation}
where $f > 0$ and $G > 0$. {Since $\varphi^{\prime}(0) = 0$, we have $g(\varphi^{\prime}(0)^{2}) = g(0)$ and $G(\varphi^{\prime}(0)^{2}) = G(0)$, so that}
\begin{equation}\label{eq:2.38}
G(0)\varphi^{\prime\prime}(0) + (n-1)g(0)\lim_{r \to 0}\frac{\varphi^{\prime}(r)}{r} = f(\varphi(0))G(0).
\end{equation}
By l'H\^opital's rule,
\begin{equation}\label{eq:2.39}
\lim_{r \to 0}\frac{\varphi^{\prime}(r)}{r} = \varphi^{\prime\prime}(0).
\end{equation}
Combining \eqref{eq:2.38} and \eqref{eq:2.39}:
\[
\varphi^{\prime\prime}(0) = \frac{f(\varphi(0))G(0)}{(n-1)g(0) + G(0)} > 0.
\]

We assume by contradiction the existence of a smallest $r_{0}$ for which $\varphi^{\prime\prime}(r_{0}) = 0$. Note that
\[
\varphi^{\prime}(r) > 0, \quad r \in (0,r_{0}], \qquad \varphi^{\prime\prime}(r) > 0, \quad r \in [0,r_{0}).
\]
Differentiating \eqref{eq:2.37}{, where $g^{\prime} = g^{\prime}(\varphi^{\prime 2})$ and $G^{\prime} = G^{\prime}(\varphi^{\prime 2})$},
\[
2G^{\prime}\varphi^{\prime}(\varphi^{\prime\prime})^{2} + G\varphi^{\prime\prime\prime} + 2(n-1)g^{\prime}\varphi^{\prime}\varphi^{\prime\prime}\frac{\varphi^{\prime}}{r} + (n-1)g\left(\frac{\varphi^{\prime\prime}}{r} - \frac{\varphi^{\prime}}{r^{2}}\right) = f^{\prime}\varphi^{\prime}G + 2fG^{\prime}\varphi^{\prime}\varphi^{\prime\prime}.
\]
At $r = r_{0}${, using $\varphi^{\prime\prime}(r_{0}) = 0$},
\[
G(\varphi^{\prime}(r_{0})^{2})\varphi^{\prime\prime\prime}(r_{0}) = (n-1)g{(\varphi^{\prime}(r_0)^2)}\frac{\varphi^{\prime}(r_{0})}{r_{0}^{2}} + f^{\prime}(\varphi(r_{0}))\varphi^{\prime}(r_{0})G(\varphi^{\prime}(r_{0})^{2}) > 0,
\]
Hence $\varphi^{\prime\prime\prime}(r_{0}) > 0$, contradicting $\varphi^{\prime\prime\prime}(r_{0}) \leq 0$, since $r_{0}$ is the first zero of $\varphi^{\prime\prime}$ and $\varphi^{\prime\prime} > 0$ on $[0,r_{0})$. Thus $\varphi^{\prime\prime} > 0$ {on $[0,R)$}.

{By the definition \eqref{eq:1.6b}, $v(r) = \int_{\varphi(r)}^{0}\frac{dy}{\sqrt{F(y)}}$, so that}
\[
{v^{\prime} = -\frac{\varphi^{\prime}}{\sqrt{F(\varphi)}}, \qquad
v^{\prime\prime} = -\frac{\varphi^{\prime\prime}F(\varphi) + \tfrac{1}{2}f(\varphi)(\varphi^{\prime})^{2}}{F(\varphi)\sqrt{F(\varphi)}} < 0,}
\]
{since $\varphi^{\prime\prime} > 0$, $f > 0$ and $F > 0$. Hence $v$, as a function of $r$, is strictly concave in $[0,R)$; moreover, the Hessian of the radial function $v$ has eigenvalues $v^{\prime\prime}(r) < 0$ and $v^{\prime}(r)/r = -\varphi^{\prime}(r)/(r\sqrt{F(\varphi)}) < 0$ (with multiplicity $n-1$), so $v$ is strictly concave in $B_{R}(0)$.}
\end{proof}

\subsection{Proof of Theorem~\ref{t1}}

We first state the following known Boundary Convexity Lemma (see \cite{CF85} or \cite{KL87}).

\begin{lemma}\label{l23}
Let $\Omega \subset \mathbb{R}^{n}$ be a smooth bounded and strictly convex domain (i.e., all principal curvatures of $\partial\Omega$ are positive). Let $u \in C^{\infty}(\Omega) \cap C^{1,1}(\overline{\Omega})$ satisfy
\[
u < 0 \text{ in } \Omega, \qquad u = 0, \quad \frac{\partial u}{\partial \mathbf{n}} > 0 \text{ on } \partial\Omega,
\]
where $\mathbf{n}$ is the unit exterior normal to $\partial\Omega$. Let
\[
\Omega_{\epsilon} = \{x \in \Omega : d(x, \partial\Omega) > \epsilon\},
\]
and $v = \phi(u)$. Then, for small enough $\epsilon$, the function $v$ is strictly concave in the boundary strip $\Omega \setminus \Omega_{\epsilon}$, provided $\phi$ satisfies
\[
\phi^{\prime} < 0, \qquad \phi^{\prime\prime} < 0, \qquad \lim_{u \to 0^{-}}\frac{\phi^{\prime}}{\phi^{\prime\prime}} = 0.
\]
\end{lemma}

We now use the deformation technique combined with the constant rank theorem (see the same technique in \cite{KL87, MX08, LMX10}). For an arbitrary strictly convex domain $\Omega$, set $\Omega_{t} = (1-t)B + t\Omega$, $0 \leq t \leq 1$, where $B$ is the unit ball. From the theory of convex bodies (\cite{Sc14}), we can deform $B$ continuously into $\Omega$ by a family $\{\Omega_{t}\}$, $0 \leq t < 1$, of strictly convex domains such that $\partial\Omega_{t} \to \partial\Omega_{s}$ as $t \to s$ in the sense of Hausdorff distance, whenever $0 \leq s \leq 1$. Furthermore, $\partial\Omega_{t}$ can be locally represented by a function whose $C^{2,\alpha}$ norm, $0 < \alpha < 1$, depends only on $\delta$, whenever $0 < t \leq 1$.

Let $u_{t} \in C^{\infty}(\Omega_{t})$ be the solution of problem \eqref{eq:1.4}--\eqref{eq:1.5} on $\Omega_{t}$, $v_{t} := v(u_{t})$ be the function defined by \eqref{eq:1.6}--\eqref{eq:1.6b}, and $H_{t}$ be the corresponding Hessian matrix of $v_{t}$. Note that $H_{0}$ is {negative} definite, and from Lemma~\ref{l23}, $H_{\delta}$ is {negative} definite in an $\epsilon$-neighborhood of $\partial\Omega_{\delta}$. From the a priori estimates on the solution $u$ of problem \eqref{eq:1.4}--\eqref{eq:1.5} (\cite{GT77}), we know this bound depends only on the uniformly bounded geometry of $\Omega_{t}$, which depends on $\Omega$ and $t$. We conclude that if $v(\cdot,s)$ is strictly {concave} for all $0 \leq s < t$, then $v(\cdot,t)$ is {concave}. Therefore, if for some $\delta$, $0 < \delta < 1$, the matrix $H_{\delta}$ is {negative} semi-definite but not {negative} definite in $\Omega_{\delta}$, then this is impossible by the constant rank theorem (Theorem~\ref{thm:constant-rank}) and boundary estimates (Lemma~\ref{l23}). We conclude that $H_{\delta}$ is {negative} definite. Thus, the function $v$ defined in \eqref{eq:1.6}--\eqref{eq:1.6b} is strictly concave in $\Omega$.

\section{Proof of Theorem~\ref{t2}}

Since the right-hand side of \eqref{eq:1.4} is positive, the strong maximum principle implies that $u < 0$ in $\Omega$, so $u$ attains its minimum at some interior point of $\Omega$. Moreover, since $v$ defined in \eqref{eq:1.6}--\eqref{eq:1.6b} is strictly concave, there is a unique interior minimum point for the solution of problem \eqref{eq:1.4}--\eqref{eq:1.5}.

\begin{lemma}\label{l31}
If $\beta \in (1,2)$, then the auxiliary function $\Phi(\mathbf{x};\beta)$ attains its minimum value at the critical point of $u$ or at some point of the boundary $\partial\Omega$.
\end{lemma}

\begin{proof}
Differentiating \eqref{eq:1.8}:
\begin{equation}\label{eq:3.1}
\Phi_{k} = 2u_{ik}u_{i} - \beta f u_{k},
\end{equation}
and
\begin{equation}\label{eq:3.2}
\Phi_{kl} = 2(u_{ikl}u_{i} + u_{ik}u_{il}) - \beta f^{\prime}u_{k}u_{l} - \beta f u_{kl}.
\end{equation}

We recall the following inequality due to Philippin and Safoui~\cite{PS04}:
\begin{equation}\label{eq:3.3}
u_{ik}u_{ik}|\nabla u|^{2} \leq |\nabla u|^{2}(\Delta u)^{2} + 2u_{ij}u_{i}u_{kj}u_{k} - 2(\Delta u)u_{ij}u_{i}u_{j}.
\end{equation}

Using \eqref{eq:3.1}--\eqref{eq:3.3} and equation \eqref{eq:1.4}, after some manipulations we obtain (see also \cite{PS04}):
\begin{equation}\label{eq:3.4}
\begin{aligned}
L\Phi &:= g\Delta\Phi + 2g^{\prime}\Phi_{kj}u_{k}u_{j} + W_{k}\Phi_{k} \\
&= (\beta - 2)\frac{G}{g}\left( (\beta - 1)Gf^{2} - \left( f^{\prime}g + \beta\,\frac{G^{\prime}}{G}gf^{2}\right) \left\vert \nabla u\right\vert^{2}\right),
\end{aligned}
\end{equation}
where $W_{k}$ is a vector function, smooth away from the critical point of $u$ and singular there. The right-hand side of \eqref{eq:3.4} is non-positive due to our assumption on the data. The conclusion follows from the strong maximum principle.
\end{proof}

\begin{lemma}\label{l32}
If $\beta \in [1,2]$, then the auxiliary function $\Phi(\mathbf{x};\beta)$ cannot be identically constant on $\overline{\Omega}$.
\end{lemma}

\begin{proof}
If $\beta \in [1,2)$, then no constant $\Phi(\mathbf{x};\beta)$ can satisfy \eqref{eq:3.4}, since by Lemma~\ref{l31} its right-hand side is non-positive and does not vanish identically, whereas $L\Phi \equiv 0$ for a constant $\Phi$. It remains to investigate $\beta = 2$. Assume that $\Phi(\mathbf{x};2)$ is constant on $\overline{\Omega}$. By the definition of $\Phi(\mathbf{x};2)$ and $u = 0$ on $\partial\Omega$, we deduce that $|\nabla u|$ is constant on $\partial\Omega$. By Serrin's symmetry result~\cite{Se71}, $\Omega$ must be a ball and the solution must be radial: $u = u(r)$, $r = |\mathbf{x}|$. In radial coordinates, equation \eqref{eq:1.4} becomes
\[
u_{rr} + \frac{n-1}{r}\frac{g}{G}u_{r} = f.
\]
Since $\Phi(\mathbf{x};2)$ is constant, $\partial\Phi/\partial r = \Phi_{,k}u_{,k} = 0$, so $u_{rr} = f(u)$ and \eqref{eq:3.4} becomes
\[
\frac{(n-1)g}{rG}u_{r}^{2} = 0,
\]
which is impossible since $u_{r} \neq 0$ for $r \neq 0$. This contradiction completes the proof.
\end{proof}

We now prove Theorem~\ref{t2} by contradiction. Assume that the minimum of $\Phi(\mathbf{x};\beta)$ occurs at the critical point $\mathbf{O}$ of $u$. We choose coordinates so that $\mathbf{O}$ is at the origin. Then $\mathbf{O}$ is the unique global minimum point of $u$, and by choosing appropriate rotations:
\begin{equation}\label{eq:3.7}
u_{i}(\mathbf{O}) = 0, \qquad u_{ij}(\mathbf{O}) = 0 \ \ (i \neq j), \qquad u_{ii}(\mathbf{O}) > 0, \qquad i,j = 1,\ldots,n.
\end{equation}

We distinguish two cases.

\textbf{Case 1:} $\beta \in (1,2]$. Using \eqref{eq:3.7}, we evaluate \eqref{eq:3.1} and \eqref{eq:3.2} at $\mathbf{O}$:
\[
\Phi_{i}(\mathbf{O};\beta) = 0, \quad i = 1,\ldots,n,
\]
\[
\Phi_{ii}(\mathbf{O};\beta) = 2u_{ii}^{2}(\mathbf{O}) - \beta f(u_{\min})u_{ii}(\mathbf{O}), \qquad \Phi_{ij}(\mathbf{O};\beta) = 0, \quad i \neq j.
\]
Since $\Phi(\mathbf{x};\beta)$ attains its minimum at $\mathbf{O}$:
\begin{equation}\label{eq:3.10}
0 \leq \Phi_{ii}(\mathbf{O};\beta) = u_{ii}(\mathbf{O})(2u_{ii}(\mathbf{O}) - \beta f(u_{\min})), \quad i = 1,\ldots,n.
\end{equation}
From \eqref{eq:3.7} and \eqref{eq:3.10}:
\[
2u_{ii}(\mathbf{O}) - \beta f(u_{\min}) \geq 0, \quad i = 1,\ldots,n.
\]
Summing over $i = 1,\ldots,n$:
\begin{equation}\label{eq:3.12}
2\Delta u(\mathbf{O}) - n\beta f(u_{\min}) \geq 0.
\end{equation}
Evaluating \eqref{eq:1.4} at $\mathbf{O}$: $\Delta u(\mathbf{O}) = f(u_{\min})$. Inserting into \eqref{eq:3.12}: $\beta \leq 2/n$, which contradicts $\beta > 1$ for every $n \geq 2$.

\textbf{Case 2:} $\beta = 1$. We use a continuity argument from \cite{PS04} (see also \cite{BE13}). From Lemma~\ref{l31}, for all $\beta \in [1,2]$, $\Phi(\mathbf{x};\beta)$ takes its minimum either on $\partial\Omega$ or at the critical point of $u$. From Case~1, $\Phi(\mathbf{x};\beta)$ takes its minimum on $\partial\Omega$ for all $\beta \in (1,2)$. As $\beta$ decreases continuously from $2$ to $1$, the minimum points move continuously and cannot jump from $\partial\Omega$ to the interior critical point. This contradiction proves Theorem~\ref{t2} for $\beta = 1$.

\section{Proof of Theorem~\ref{t3}}

From Theorem~{\ref{t2} (with $\beta = 1$)}, $\Phi(\mathbf{x};1)$ takes its minimum value at some point $\mathbf{Q} \in \partial\Omega$. This implies
\begin{equation}\label{eq:4.1}
\left\vert \nabla u\right\vert^{2} - u \geq q_{m}^{2},
\end{equation}
where $q_{m}$ is the minimum value of $|\nabla u|$ on $\partial\Omega$. Evaluating \eqref{eq:4.1} at the unique minimum point of $u$:
\begin{equation}\label{eq:4.2}
-u_{\min} \geq q_{m}^{2}.
\end{equation}

Next, we construct a lower bound for $q_{m}$ in terms of the mean curvature $K(s)$ of $\partial\Omega$. Since $\Phi(\mathbf{x};1)$ takes its minimum at $\mathbf{Q}$, we have $\partial\Phi(\mathbf{x};1)/\partial\mathbf{n} \leq 0$ at $\mathbf{Q}$, i.e.,
\begin{equation}\label{eq:4.3}
2u_{n}u_{nn} - u_{n} \leq 0 \quad \text{at } \mathbf{Q},
\end{equation}
where $u_{n}$ and $u_{nn}$ are the first and second outward normal derivatives of $u$ on $\partial\Omega$. Since $u < 0$ in $\Omega$ and $u = 0$ on $\partial\Omega$, we have $u_{n} > 0$ and $u_{n} = |\nabla u|$ on $\partial\Omega$. Thus \eqref{eq:4.3} becomes
\begin{equation}\label{eq:4.4}
u_{nn} \leq \frac{1}{2} \quad \text{at } \mathbf{Q}.
\end{equation}

{Using the standard expansion of the mean-curvature operator,
\[
\mathrm{div}\left( \frac{\nabla u}{\sqrt{1+|\nabla u|^{2}}}\right) = \frac{(1+|\nabla u|^{2})\Delta u - u_{i}u_{j}u_{ij}}{(1+|\nabla u|^{2})^{3/2}},
\]
we choose at $\mathbf{Q}$ an orthonormal frame whose last vector is the exterior unit normal $\mathbf{n}$ and whose remaining vectors are principal directions of $\partial\Omega$. Since $u = 0$ on $\partial\Omega$, we have $\nabla u = u_{n}\mathbf{n}$, $\Delta u = u_{nn} + (n-1)Ku_{n}$ and $u_{i}u_{j}u_{ij} = u_{n}^{2}u_{nn}$ at $\mathbf{Q}$ (see, e.g., \cite{En05thesis}), where $K$ denotes the mean curvature of $\partial\Omega$. Substituting into \eqref{eq:1.9} and multiplying by $(1+u_{n}^{2})^{3/2}$ gives}
\begin{equation}\label{eq:4.6}
u_{nn} + (n-1)Ku_{n}(1+u_{n}^{2}) = 1 \quad \text{on } \partial\Omega.
\end{equation}

Inserting \eqref{eq:4.6} into \eqref{eq:4.4}:
\[
1 \leq 2(n-1)K(\mathbf{Q})q_{m}(1+q_{m}^{2}).
\]
Therefore,
\[
q_{m}(1+q_{m}^{2}) \geq \frac{1}{2(n-1)K(\mathbf{Q})} \geq \frac{1}{2(n-1)K_{\max}},
\]
or equivalently,
\begin{equation}\label{eq:4.9}
q_{m}^{3} + q_{m} - \alpha \geq 0,
\end{equation}
where $\alpha = \frac{1}{2(n-1)K_{\max}}$. Solving the cubic inequality \eqref{eq:4.9}:
\begin{equation}\label{eq:4.10}
q_{m} \geq \sqrt[3]{\frac{\alpha}{2} + \sqrt{\frac{\alpha^{2}}{4} + \frac{1}{27}}} + \sqrt[3]{\frac{\alpha}{2} - \sqrt{\frac{\alpha^{2}}{4} + \frac{1}{27}}}.
\end{equation}

Inserting \eqref{eq:4.10} into \eqref{eq:4.2} yields \eqref{eq:1.10}.

\begin{remark}
Regarding the optimality of the bound in Theorem~\ref{t3}: equality holds when the corresponding $P$-function is identically constant. However, Lemma~\ref{l32} shows this is impossible. Therefore, the estimate \eqref{eq:1.10} is not optimal. As for the bound in Theorem~\ref{thm:upper}, \eqref{eq:5.4} is sharp in the limit as $\Omega$ degenerates into a strip region of width $2d$.
\end{remark}

\section{A complementary upper bound for the Euclidean problem}

For the Euclidean problem \eqref{eq:1.9}, the maximum principle of Payne and Philippin \cite{PP79} for the corresponding $P$-function yields an upper estimate complementary to the lower bound in Theorem~\ref{t3}. For this problem $f \equiv 1$, so $F(u) = -u$, and the $P$-function \eqref{eq:1.8} with $\beta = 2$ is
\[
\Phi(\mathbf{x};2) = |\nabla u|^{2} - 2u.
\]
By \cite[Corollary~1]{PP79}, $\Phi(\mathbf{x};2)$ attains its maximum value at the unique critical point of $u$, which gives
\begin{equation}\label{eq:5.1}
|\nabla u|^{2} \leq 2u - 2u_{\min}.
\end{equation}

Let $\mathbf{P}$ be a point where $u = u_{\min}$ and $\mathbf{Q}$ be a point on $\partial\Omega$ nearest to $\mathbf{P}$. Let $r$ measure the distance from $\mathbf{P}$ to $\mathbf{Q}$ along the ray connecting them. Clearly,
\begin{equation}\label{eq:5.2}
\frac{du}{dr} \leq \left\vert \nabla u\right\vert.
\end{equation}

Integrating \eqref{eq:5.2} from $\mathbf{P}$ to $\mathbf{Q}$ and using \eqref{eq:5.1}:
\begin{equation}\label{eq:5.3}
\int_{u_{\min}}^{0}\frac{du}{\sqrt{u - u_{\min}}} \leq \sqrt{2}\int_{\mathbf{P}}^{\mathbf{Q}}dr = \sqrt{2}\left\vert \mathbf{PQ}\right\vert \leq \sqrt{2}\,d,
\end{equation}
where $d$ is the radius of the largest ball inscribed in $\Omega$. Evaluating the integral in \eqref{eq:5.3}, we obtain the following bound.

\begin{theorem}\label{thm:upper}
Let $d$ be the radius of the largest ball inscribed in $\Omega$. If $u$ is the solution of problem \eqref{eq:1.9}, then
\begin{equation}\label{eq:5.4}
-u_{\min} \leq \frac{d^{2}}{2}.
\end{equation}
\end{theorem}

\section{Declarations}

\subsection{Funding}
Rafael L\'opez has been partially supported by MINECO/\discretionary{}{}{}MICINN/\discretionary{}{}{}FEDER grant no.\ PID2023-150727NB-I00, and by the ``Mar\'{\i}a de Maeztu'' Excellence Unit IMAG, reference CEX2020-001105-M, funded by MCINN/\discretionary{}{}{}AEI/\discretionary{}{}{}10.13039/\discretionary{}{}{}501100011033/\discretionary{}{}{}CEX2020-001105-M.

\subsection{Conflicts of interest}
The authors have no relevant financial or non-financial interests to disclose.

\end{document}